\documentclass{amsart}

\usepackage{amsthm}
\usepackage{thmtools}
\usepackage{amsfonts}
\usepackage{amssymb}
\usepackage{tikz-cd}
\usepackage{xcolor}
\usepackage{todonotes}
\setuptodonotes{inline}
\usepackage{bm}
\usepackage{hyperref}
\usepackage[shortlabels]{enumitem}

\definecolor{Dark Orchid}{RGB}{153,50,204}

\newcommand{\tuple}[1]{\overline{#1}}

\newcommand{\cat}[2]{#1\widehat{\phantom{\alpha}}\str{#2}}

\newcommand{\proves}{\vdash}

\newcommand{\str}[1]{\left\langle #1 \right\rangle}
\newcommand{\ca}[2]{#1\widehat{\phantom{\alpha}}#2}
\newcommand{\catt}[2]{#1\widehat{\phantom{\alpha}}\str{#2}}
\newcommand{\catttt}[4]{#1\widehat{\phantom{\alpha}}\str{#2}%
	\widehat{\phantom{\alpha}}\str{#3}\widehat{\phantom{\alpha}}\str{#4}}
\newcommand{\cattttt}[5]{#1\widehat{\phantom{\alpha}}#2%
	\widehat{\phantom{\alpha}}\str{#3}\widehat{\phantom{\alpha}}\str{#4}%
	\widehat{\phantom{\alpha}}\str{#5}}

\newcommand{\restr}{\restriction}
\newcommand\lex{\mathrm{lex}}

\newcommand{\cC}{\mathcal{C}}
\newcommand{\cD}{\mathcal{D}}

\newcommand{\cI}{\mathcal{I}}

\newcommand{\cL}{\mathcal{L}}
\newcommand{\cP}{\mathcal{P}}

\newcommand{\cS}{\mathcal{S}}
\newcommand{\cT}{\mathcal{T}}
\newcommand{\cU}{\mathcal{U}}

\newcommand{\bP}{\mathbb{P}}

\DeclareMathOperator{\ran}{\operatorname{ran}}

\DeclareMathOperator{\dom}{\operatorname{dom}}

\theoremstyle{plain}
\newtheorem{theorem}{Theorem}[section]

\newtheorem{prop}[theorem]{Proposition}
\newtheorem{lemma}[theorem]{Lemma}
\newtheorem{claim}[theorem]{Claim}

\newtheoremstyle{noname}{}{}{\itshape}{}{\bfseries}{.}{.5em}{\thmnote{#3}}
\theoremstyle{noname}
\newtheorem*{nonamethm}{Theorem}
\theoremstyle{definition}
\newtheorem{defn}[theorem]{Definition}

\newtheorem{rem}[theorem]{Remark}
\newtheorem{notation}[theorem]{Notation}

\numberwithin{equation}{section}

\begin{document}
	
\title{The Theory of the Ziegler degrees}
	
\author[Downey]{Rodney G.~Downey}
\address{School of Mathematics and Statistics\\
	Victoria University of Wellington\\
	P.O.~Box 600\\
	Wellington 6140\\
	NEW ZEALAND}
\email{\href{mailto:rod.downey@vuw.ac.nz}{rod.downey@vuw.ac.nz}}
\urladdr{\url{https://homepages.ecs.vuw.ac.nz/~downey/}}
	
\author[Lempp]{Steffen Lempp}
\address{Department of Mathematics\\
	University of Wisconsin-Madison\\
	Madison, WI 53706-1325\\
	USA}
\email{\href{mailto:lempp@math.wisc.edu}{lempp@math.wisc.edu}}
\urladdr{\url{http://www.math.wisc.edu/~lempp}}
	
\author[Scott]{Isabella Scott}
\address{Department of Mathematics\\
University of Wisconsin-Madison\\
		Madison, WI 53706-1325\\
	USA}
\email{\href{mailto:iscott6@wisc.edu}{iscott6@wisc.edu}}
\urladdr{\url{https://people.math.wisc.edu/~iscott6/}}
	
\thanks{The first author's research was partially supported by Marsden Grant
20-VUW-105 The second author's research was partially supported by AMS-Simons
Foundation Collaboration Grant 626304.}

	
	
\begin{abstract}
The Ziegler degrees were introduced to characterize definability in group
theory. In~\cite{JLSta}, the authors show that the first order-theory of
the Ziegler degrees (as a partial order) is undecidable. We improve this
result, showing that the theory of the Ziegler degrees is bi-interpretable
with true second-order arithmetic.
\end{abstract}

\maketitle
	
\section{Introduction}

Ziegler reducibility arose in Martin Ziegler's
investigation into the structure of existentially closed groups~\cite{Zi80}.
In this landmark study, he showed that Ziegler reducibility between word
problems of finitely generated groups characterizes quantifier-free parameter
definability, which in turn characterizes relative omitting types in
existentially closed groups. In addition, recent work shows that Ziegler
reducibility also characterizes a notion of definability in
subshifts~\cite{NCS26}.

The restriction of Ziegler reducibility to c.e.\ sets had appeared in its own
right earlier under the name of \emph{quasireducibility} or
\emph{Q-reducibility}. Rogers attributes the definition of Q-reducibility to
Tennenbaum (see \cite{Ro67}, p.~159), who studied it in the context of Post's
problem. Belegradek formulated the analogue of Ziegler's theorem for finitely
presented groups, which is in terms of Q-reducibility, in \cite{Be74}. The
degree theory of Q-reducibility has attracted renewed attention in recent
years, for example in \cite{BDS26,KNta}.
\smallskip

We postpone a formal definition of Ziegler reducibility until
Definition~\ref{def:Z-red}, but we describe the motivation here. Ziegler
called the reducibility ``star-reducibility'' or~``$\leq^*$''
(\cite[Def.~III.1.1]{Zi80}). Here we denote it as~$\leq_Z$ in analogy with
standard notation for reducibilities and in recognition of growing interest
around it. The definition is unusual in that it is a ``total reducibility''
--
in the sense that $A \leq_Z B$ implies that both the positive ($a \in A$) and
the negative ($a \notin A$) information about~$A$ are ``computed'' from
positive and negative information about the oracle~$B$ -- but the
reducibility is not symmetric in the positive and negative information. The
positive information $a \in A$ is determined by finitely many positive bits
of~$B$ (via \emph{enumeration reducibility}, which will also be defined in
Definition~\ref{def:Z-red}). On the other hand, the negative information $a
\notin A$ is determined by finitely many positive bits and a \emph{single
negative bit}. This reflects the syntactic provability relation for atomic
sentences in group theory over some parameter set: Let~$U$ and~$V$ be sets of
words, and $T = \{ u(\tuple g) = 1, v(\tuple g) \neq 1 \mid u \in U, v \in V
\}$ be a set of atomic sentences in the letters~$\tuple g$. Let $T^+ = \{
u(\tuple g) = 1 \mid u \in U \}$ be the \emph{positive part} of~$T$ and $T^-
= \{ v(\tuple g) \neq 1 \mid v \in V \}$ be the \emph{negative part} of~$T$.
Then for any word $w(\tuple g)$, we have
\begin{align*}
T \proves w(\tuple g) = 1 &\iff T^+ \proves w(\tuple g) = 1\\
   	&\iff\ \text{there is a finite subset}\ T^+_0 \subseteq T^+\
    \text{with}\ T^+_0 \proves w(\tuple g) = 1.
\end{align*}
On the other hand, for any word $w(\tuple g)$, we have
\begin{align*}
T \proves w(\tuple g) \neq 1 &\iff\ \text{for some}\ v(\tuple g)
    \neq 1 \in T^{-},\; T^+ \cup \{ v(\tuple g) \neq 1\} \proves w(\tuple g)
    \neq 1\\
	&\iff \text{there is a finite}\ T_0^+ \subseteq T^+\ \text{and}\
    v(\tuple g) \neq 1 \in T^-\ \text{with}\\
	&\qquad\qquad T_0^+ \cup \{v(\tuple g) \neq 1\}
    \proves w(\tuple g) \neq 1.
\end{align*}


The structure of Ziegler reducibility parallels this relationship.
\smallskip



In this paper, we investigate the complexity of the first-order theory of the
Ziegler degrees. This was undertaken for the Turing degrees in classical work
of Simpson, in which he showed that the theory is computably isomorphic to
true second-order arithmetic~\cite{Si77}. Since the poset of the Turing
degrees is interpretable in second-order arithmetic, this is a priori as
complicated as possible. The first-order theory of the enumeration degrees is
also recursively isomorphic to true second-order arithmetic, which again is
as complicated as possible~\cite{SW97}. In~\cite{JLSta}, the authors show
that the first-order theory of the Ziegler degrees, is undecidable -- indeed,
they show this undecidability already holds for the
$\forall\exists\forall$-theory. Write~$\cD_Z$ for the poset of Ziegler
degrees. In this paper, we strengthen this result by showing that the
first-order theory of~$\cD_Z$ is recursively isomorphic to true second-order
arithmetic.

Following the framework of Nerode and Shore \cite{NS80}, we prove two
theorems about the algebraic structure of the poset of the Ziegler degrees;
namely, that every countable ideal has an exact pair and that every countable
distributive lattice embeds as an initial segment.
	
\begin{theorem}\label{thm:expair}
Every countable ideal~$\cI$ in~$\cD_Z$ has an exact pair; i.e., a pair of
degrees $\deg(A)$ and $\deg(B)$ such that $\deg(C) \in \cI$ iff $C \leq_Z A$
and $C \leq_Z B$.
\end{theorem}
	
\begin{theorem}\label{thm:ctdl}
Every countable distributive lattice embeds as an initial segment in~$\cD_Z$.
\end{theorem}

The main theorem follows from these results via the work of~\cite{NS80}.
	
\begin{theorem}[Main Theorem]\label{thm:secarith}
$Th(\cD_Z)$ is bi-interpretable with true second-order arithmetic.		
\end{theorem}

The proof of Theorem~\ref{thm:expair} is essentially identical to the
argument of the same result in the Turing degrees, in~\cite{Sp56}, and is
presented in Section~\ref{sec:expair}. The proof of Theorem~\ref{thm:ctdl},
given in Section~\ref{sec:latt}, occupies the majority of the paper. Our
proof follows the outline of \cite{La68}, which first proved the
corresponding result in the Turing degrees. However, we make several
modifications to guarantee that the reductions are Ziegler reductions. We
start with the relevant definitions in Section~\ref{sec:notation}.
	


\section{Notation}\label{sec:notation}

Let~$D_u$ be the $u$th canonical finite set, and~$(W_e)_e$ a computable
enumeration of the c.e.\ sets.
	
\begin{defn}\label{def:Z-red}
Let $A, B \subseteq \omega$. Then:
\begin{enumerate}
\item\label{item:pos-red}
$A$ is \emph{enumeration reducible} to~$B$, written $A \leq_eB$, if there
is a c.e.\ set~$W_e$ such that
\[
a \in A \iff \exists \langle a,u \rangle \in W_e\; (D_u \subseteq B).
\]
In this case, we write $A = \Phi_e[B]$.
\item\label{item:neg-red}
$A \leq_e^1 B$ if there is a c.e.\ set~$W_i$ such that
\[
a \in A \iff \exists \langle a, u, v \rangle \in W_i\; (D_u \subseteq B
  \land D_v \cap B = \emptyset \land |D_v| \leq 1).
\]
In this case, we write $A = \Psi_i[B]$.
\item\label{item:Z-red}
$A$ is \emph{Ziegler reducible} to~$B$, written $A \leq_Z B$, if
\begin{enumerate}
\item $A \leq_e B$, and
\item$\overline{A} \leq_e^1 B$.
\end{enumerate}

If $A = \Phi_e[B]$ and $\overline{A} = \Psi_i[B]$, then we write $A =
(\Phi_e, \Psi_i)[B]$.
\end{enumerate}
\end{defn}

Compare the preceding definition to the motivation for Ziegler reducibility
in the introduction:~\eqref{item:pos-red} captures that positive information
about~$A$ is determined by positive information about~$B$;
and~\eqref{item:neg-red} (applied to~$\overline A$) captures that negative
information about~$A$ is determined by positive information and ``one bit of
negative information at a time'' about~$B$.

It is not hard to check that~$\leq_Z$ is reflexive and transitive and hence
induces an equivalence relation on~$2^\omega$, which we denote by~$\equiv_Z$.
We write~$\cD_Z$ for the induced partial order $2^\omega / \equiv_Z$.

Notice that Ziegler reducibility is a common refinement of both enumeration-
and Turing-reducibility, and is in turn refined by $m$-reducibility. It is
not hard to show, although we do not do it here, that all of these
refinements are strict.

	
\begin{prop}\label{prop:usl}
$\cD_Z$ is an upper semilattice with joins given by $A \oplus B = \{2a \mid a
\in A\} \cup \{2b+1 \mid b \in B\}$ and least element~$\mathbf{0}$.
Furthermore,~$\mathbf{0}$ consists exactly of the computable sets.
\end{prop}
	
	
\section{\texorpdfstring{The degree of $\cD_Z$}{The degree of DZ}}%
\label{sec:degD*}


	
	
	

We recall the main theorem of this paper:

\begin{nonamethm}[Theorem~\ref{thm:secarith}]
$Th(\cD_Z)$ is bi-interpretable with true second-order arithmetic.
\end{nonamethm}

Our proof of Theorem~\ref{thm:secarith} is based on the following classical
result of Nerode and Shore.
	
\begin{theorem}[{Nerode/Shore \cite[Theorem~3.2 and Remark~3.3]{NS80}}]%
\label{thm:NS}
If~$\cP$ is a partial order satisfying the following three conditions:
\begin{align}
&\text{$\cP$ is an upper semilattice with least element~0},\label{3.1}\\
&\text{Every countable ideal in~$\cP$ has an exact pair, and}\label{3.2}\\
\begin{split}
&\text{Every countable distributive lattice~$\cL$ with a least element}\\
&\qquad\qquad\text{ is isomorphic to some initial segment of~$\cP$,}
\end{split}\label{3.3}
\end{align}
then true full second-order arithmetic is $1$-reducible to the first-order
theory of~$\cP$.
\end{theorem}
	
\begin{proof}[Proof of Theorem~\ref{thm:secarith}]
Clearly the first-order theory of~$\cD_Z$ is $1$-reducible to second-order
arithmetic. The other direction follows immediately by Theorem~\ref{thm:NS}:
Condition~\eqref{3.1} that~$\cD_Z$ is an upper semilattice is
Proposition~\ref{prop:usl}, and Conditions~\eqref{3.2} and~\eqref{3.3} are
Theorems~\ref{thm:expair} and~\ref{thm:ctdl} above, respectively.
\end{proof}
	
\section{An exact pair theorem for the Ziegler degrees}\label{sec:expair}

In this section, we prove an exact pair theorem for the Ziegler degrees. Our
proof follows Spector~\cite{Sp56}, who proved the corresponding result for
the Turing degrees. Case~\cite[Theorem~4.3]{Ca70} showed that a similar proof
works for the enumeration degrees.
	
\begin{nonamethm}[Theorem~\ref{thm:expair}]
Every countable ideal~$\cI$ in~$\cD_Z$ has an exact pair; i.e., a pair of
degrees~$\deg(A)$ and~$\deg(B)$ such that $\deg(C) \in \cI$ iff $C \leq_Z A$
and $C \leq_Z B$.
\end{nonamethm}
	
\begin{proof}
Let~$\cI$ be a countable ideal in~$\cD_Z$. Let $A_0 \leq_Z A_1 \leq_Z \cdots
$ represent a cofinal sequence of degrees in~$\cI$. Let $A = \bigoplus_n A_n$
(viewing~$A$ as a function from~$\omega$ to~$\{0,1\}$). We will build~$B$ to
satisfy the following requirements:
\[
R_{\str{e,i,j,k}}: (\Phi_e, \Psi_i)[A] = (\Phi_j, \Psi_k)[B] = X \implies
  \exists n\left( X \leq_Z A_n \right)
\]
and ensure at the same time that $A_n \le_Z B$ for all~$n$.

(We note that Posner's trick does not seem to work in the Ziegler degrees --
one would need to check negative information about the oracle in order to
determine which computation to follow, which breaks Ziegler reducibility's
very parsimonious use of negative information.)

The strategy is to code each~$A_n$ into the $n$th column of~$B$, up to a
finite set. The construction will fix~$B$ (viewed as a function from~$\omega$
to~$\{0,1\}$) on~$\omega^{[< s]}$ by stage~$s$. At stage~$s$, we will
extend~$B$ on~$\omega^{[\geq s]}$ by a finite amount to ensure that $(\Phi_e,
\Psi_i)[A] \neq (\Phi_j, \Psi_k)[B]$ if possible. If not, then, using
that~$B$ is fixed on~$\omega^{[< s]}$ and $*$-computable from~$A_s$, we will
have that if~$(\Phi_e, \Psi_i)[A]$ is a valid Ziegler reduction
then~$(\Phi_e, \Psi_i)[A] \leq_Z A_s$.
		
\subsubsection*{Construction}
		
\paragraph*{\it Stage $s=0$:}

We set $B_0 = \emptyset$.
		
\paragraph*{\it Stage $s+1$:}

Suppose that~$B_s$ has been defined such that $\dom(B_s) \supseteq
\omega^{[<s]}$ and $B_s =^* A^{[<s]}$. Consider $R_s = R_{\langle
e,i,j,k \rangle}$. If there are a finite partial function~$\tau$ and $x \in
\omega$	with $\dom(\tau) \cap \dom(B_s) = \emptyset$ and $(\Phi_e,
\Psi_i)[A](x) \neq (\Phi_j, \Psi_k)[B_s \cup \tau](x)$, then set $C_{s+1} =
B_s \cup \tau$. Else set $C_{s+1} = B_s$.
We now need to fill the rest of the~$s$th column with~$A_s$ without
disturbing the diagonalization.	Let~$A^0_s$ be the subset of the
function~$A_s$ with domain $\omega \setminus \dom(C_{s+1})^{[s]}$. Then set
$B_{s+1} = C_{s+1} \cup (\{s\} \times A_s^0)$.
		
\subsubsection*{Verification}

Firstly, since $B^{[s]} =^* A_s$, we have that $A_s \leq_Z B$ for
every~$s$ (non-uniformly in~$s$). 	

Moreover, if~$\tau$ is found at stage~$s+1$, then~$R_s$ is immediately
satisfied since $B_{s} \cup \tau$ is preserved in~$B$ and thus $(\Phi_e,
\Psi_i)[A] \neq (\Phi_j, \Psi_k)[B]$. If no~$\tau$ is found at stage~$s+1$
and $(\Phi_e, \Psi_i)[A] = (\Phi_j, \Psi_k)[B] = X$, then we argue that $X
\leq_Z A_s$. Indeed, if $x \in X = (\Phi_e, \Psi_i)[A]$ and we did not find
$\tau$ at stage $s+1$, then any convergent computation of $(\Phi_j, \Psi_k)$
from an extension of~$B_s$ must put~$x$ into~$X$. Furthermore, since $x \in
(\Phi_j,\Psi_k)[B]$, there is some such convergent computation. Thus, $X
\leq_e B_s$ via the set $\{ \langle x, D' \rangle \mid \exists \langle x,D
\rangle \in \Phi_j\; (D' = D \cap \dom(B_s)) \}$. On the other hand, $x
\notin X$ iff $\Psi_i[A](x) = 1$. Again, since no~$\tau$ was found at
stage~$s+1$, any convergent computation of $(\Phi_j,\Psi_k)$ from an oracle
extending~$B_s$ must agree with $x \notin X$; moreover, such a convergent
computation exists since $\Psi_k[B](x) \downarrow = 1$. As before, this
implies that $\overline X \leq_e^1 B_s$. Hence, using that if $Y =^*
Z$, then $Y \leq_e Z$ and $Y \leq_e^1 Z$, we have that
\[
X \leq_e B_s \equiv_e B^{[<s]} \equiv_e A_0 \oplus \cdots \oplus A_{s-1}
\]
and
\[ X
\leq_e^1 A_0 \oplus \cdots \oplus A_{s-1}.
\]
Thus,
\[
X \leq_Z A_0 \oplus \cdots \oplus A_{s-1} \leq_Z A_s,
\]
as required.
\end{proof}
	
\section{Embedding all countable distributive lattices with least element}%
\label{sec:latt}

In this section, we adapt Lachlan's argument that every countable
distributive lattice with least element embeds as an initial segment of the
degree structure~$\cD_T$ of the Turing degrees to our setting:
	
\begin{nonamethm}[Theorem~\ref{thm:ctdl}]
Every countable distributive lattice~$L$ with least element embeds as an
initial segment of~$\cD_Z$.
\end{nonamethm}
	
\begin{proof}
We follow Lachlan's proof in~\cite{La68}, using somewhat updated notation so
that the reader can follow more easily.

Let~$L$ be a countable distributive lattice with least element~$0_L$. Without
loss of generality, we may assume that it also has a greatest element~$1_L$.
We let $(L_s)_{s < \omega}$ be an increasing approximation to~$L$ such
that~$L_0$ is the two-element lattice and~$L_{s+1}$ is a \emph{simple
extension} of~$L_s$, i.e.,~$L_{s+1}$ is generated, under join and meet,
by~$L_s$ and a single new element~$a_s$.
		
\subsection{Definitions}

When considering a finite (distributive) sublattice~$L'$ of~$L$, we will fix
the following lattice-theoretic notation:
		
\begin{defn}\label{def:latt}
Let~$L'$ be a finite distributive lattice and~$a$ a (nonzero)
join-irreducible element in~$L'$. Then we let $S_a = \{b \in L' \mid b
\ngeq_{L'} a\}$. (By distributivity,~$S_a$ has a greatest element~$b_a$.)
\end{defn}

As usual, we will produce a degree $\mathbf{f}_1 = \mathbf{f}_{1_L}$,
represented by some $f_1 = f_{1_L} \in 2^\omega$, such that $\cD_Z(\leq
\mathbf{f}_1) \cong L$ by forcing with f-trees. When embedding a finite
distributive lattice~$L'$, we may fix in advance a representation of~$L'$ as
a subset of a finite Boolean algebra, thought of as a subset of~$\cP(S)$ for
some initial segment~$S$ of~$\omega$. Here, however, we will have to modify
the representation as~$L_s$ grows. This will be kept track of in the forcing
conditions.

We first recall some definitions from Jacobsen-Grocott/Lempp/Scott
\cite{JLSta}.
		
\begin{defn}\label{def:ftree}
\begin{enumerate}
\item\label{it:ftree}
An \emph{f-tree} is a map $\cT: 2^{< \omega} \to 2^{< \omega}$ such that
for all $\pi, \rho \in 2^{< \omega}$, $\pi \subset \rho$ iff $\cT(\pi)
\subset \cT(\rho)$, and $\pi <_\lex \rho$ iff $\cT(\pi) <_\lex \cT(\rho)$.
(We also allow the \emph{degenerate f-tree} $\cT: 2^{< \omega} \to 2^{<
\omega}$ where $\cT(\rho) = \str{0^{|\rho|}}$ for all $\rho \in 2^{<
\omega}$.)
\item\label{it:full}
The \emph{full binary f-tree} is the f-tree~$\cU$ defined by setting
$\cU(\rho) = \rho$ for all $\rho \in 2^{<\omega}$.
\item\label{it:paths}
We denote the set of \emph{infinite paths} through an f-tree~$\cT$ by
$[\cT] = \{f \in 2^\omega \mid \exists^\infty \rho\, (\cT(\rho) \subset
f)\}$.
\item\label{it:subftree}
A \emph{sub-f-tree} of an f-tree~$\cT$ is an f-tree~$\cS$ with $\ran(\cS)
\subseteq \ran(\cT)$. We then write $\cS \subseteq \cT$. (In particular,
the degenerate f-tree is a sub-f-tree of every f-tree~$\cT$ with $0^\omega
\in [\cT]$.)
\item\label{it:contra}
An f-tree~$\cT$ \emph{has no explicit $\str{e,i}$-contradictions} if, for
every $\tau \in \ran(\cT)$ and every $x \in \omega$, it is not the case
that $\Phi_e^\tau(x) = 1$ and $\Psi_i^\tau(x) = 1$ (i.e.,~$x$ is not
enumerated into both the set being computed and its complement).
\item\label{it:veryunif}
A \emph{very uniform} f-tree is an f-tree~$\cT$ satisfying:
\begin{itemize}
\item
For all $\pi, \rho \in 2^{< \omega}$, $|\pi| = |\rho|$ implies
$|\cT(\pi)| = |\cT(\rho)|$;
\item
For all $l \in \omega$ and $k<2$, there is~$\tau_k^l$ such that for all
$\rho \in 2^l$, $\cT(\cat \rho k) = \cat{\cT(\rho)}{\tau_k^l}$; and
\item
for all $l \in \omega$, $k<2$, and $x < |\tau^l_k|$, if $\tau^l_0(x) \neq
\tau^l_1(x)$ then $\tau^l_0(x) < \tau^l_1(x)$. (Thus, in particular, any
disagreements on any fixed~$\tau^l_k$ ``agree'' with each other.)
\end{itemize}
\item\label{it:split}
$\tau_0$ and~$\tau_1$ \emph{$\str{e,i}$-split~$\sigma$ on~$\cT$} if
$\tau_0, \tau_1, \sigma \in \ran(\cT)$; $\tau_0, \tau_1 \supset \sigma$;
and
$$
x \notin (\Phi_e, \Psi_i)[\tau_0]\ \text{and}\
x \in (\Phi_e, \Psi_i)[\tau_1].
$$
If in addition $\tau_0 = \cT(\pi_0)$, $\tau_1 = \cT(\pi_1)$, $|\pi_0| =
|\pi_1|$, and~$\pi_0$ and~$\pi_1$ differ in exactly one bit~$j$, say, then
we say that~$\tau_0$ and~$\tau_1$
\emph{minimally~$\str{e,i}$-split~$\sigma$}. In a context with a fixed~$n <
\omega$, we may in addition specify that the minimal splitting is
\emph{via~$j'$}, where~$j' \in [0,n)$ and~$j' \equiv j \pmod{n}$. (In the
proof, we will be working with a lattice embedding $\eta: L' \to
\cP([0,n))$ for some~$n$. This will be where~$n$ arises from.)
\item\label{it:splitting}
A very uniform f-tree~$\cT$ is \emph{$\str{e,i}$-splitting} if, for every
$\rho \in \dom(\cT)$, $\cT(\catt\rho0)$ and $\cT(\catt\rho1)$
$\str{e,i}$-split~$\cT(\rho)$ on~$\cT$.
\end{enumerate}
\end{defn}

We first note the following key observation.
		
\begin{lemma}\label{lem:minsplit}
Let~$\cT$ be a very uniform f-tree which has no explicit
$\str{e,i}$-contra\-dictions, and suppose there are $\tau_0 = \cT(\pi_0)$ and
$\tau_1 = \cT(\pi_1)$ which $\str{e,i}$-split~$\sigma$ on~$\cT$. Then there
are $\tau'_0 = \cT(\pi'_0)$ and $\tau'_1 = \cT(\pi'_1)$ which minimally
$\str{e,i}$-split~$\sigma$ on~$\cT$. Furthermore, if $|\tau_0| = |\tau_1|$,
then $\pi'_0$ and~$\pi'_1$ can be chosen to differ at some~$j$ at
which~$\pi_0$ and~$\pi_1$ also differ.
\end{lemma}
		
		
\begin{proof}
Suppose that~$\tau_0$ and~$\tau_1$ $\str{e,i}$-split~$\sigma$ on~$\cT$ at
some argument~$x$, say, by symmetry, $x \in \Psi_i[\tau_0]$ via a finite set
$D \subseteq \tau_0$ and some~$y$ with $\tau_0(y)=0$, and $x \in
\Phi_e[\tau_1]$ via a finite set $E	\subseteq \tau_1$. Without loss of
generality, we may assume that $|\pi_0| = |\pi_1|$ and so $|\tau_0| =
|\tau_1|$. But then $\tau_1(y)$ is defined, and furthermore, $\tau_0(y)=0$ is
caused by $\pi_0(j)=0$ for some~$j$. Now if	$\tau_1(y)=0$, then we can set
$\pi'(k) = \max(\pi_0(k), \pi_1(k))$ for all $k < |\pi_0|$, and so $\tau' =
\cT(\pi')$ will witness an $\str{e,i}$-contradiction via~$x$. Thus we must
have $\tau_1(y)=1$, and we can set $\pi'_0(k) = \max(\pi_0(k), \pi_1(k))$ for
all $k \neq j$, and $\pi'_1(k) = \max(\pi_0(k), \pi_1(k))$ for all $k <
|\pi_0|$. But then $\tau'_0 = \cT(\pi'_0)$ and $\tau'_1 = \cT(\pi'_1)$
minimally $\str{e,i}$-split~$\sigma$ on~$\cT$ as desired.
\end{proof}

We also need the following notation.

\begin{notation}\label{not:[]}
Given a string $\rho \in 2^{<\omega}$, integers $k \ge l > 0$ and a strictly
increasing sequence $(h_0, h_1, \ldots, h_{l-1})$ of integers~$<k$, we define
$\rho[k; h_0, h_1, \ldots, h_{l-1}] \in 2^{<\omega}$ by setting, for $p \in
\omega$ and $q<l$:
$$
\rho[k; h_0, h_1, \ldots, h_{l-1}](pl+q) = \rho(pk+h_q),
$$
and similarly for infinite paths~$f$.

So $\rho[k; h_0, h_1, \ldots, h_{l-1}]$ is obtained from~$\rho$ by deleting,
from every~$k$ consecutive bits, the bits not in positions $h_0, h_1, \ldots,
h_{l-1}$. Note that the map $\rho \mapsto \rho[k; h_0, h_1, \ldots, h_{l-1}]$
is surjective onto~$2^{<\omega}$.
\end{notation}
		
To illustrate,
consider the following example: Suppose $\rho = (0,1,1,1,0,1,0,0,1,0)$ and
let $(0,1,3)$ be a the strictly increasing sequence of integers. Then
\[
\rho[5;0,1,3] = (\mathbf{0},\mathbf{1},1,\mathbf{1},0,
  \mathbf{1},\mathbf{0},0,\mathbf{1},0) = (0,1,1,1,0,1).
\]
		
\subsection{The forcing partial order}
		
\begin{defn}\label{def:triple}
Let~$\bP$ be the set of pairs or \emph{conditions} $(\eta, \cT)$ where
\begin{enumerate}
\item\label{it:treta}
$\eta$ is a lattice embedding from a finite sublattice $L' \subseteq L$
into $\cP(\omega)$ (ordered by inclusion) where $\eta(0_L) = \emptyset$ and
$\eta(1_L) = [0,n)$ for some $n \in \omega$.
\item\label{it:trT}
$\cT$ is a map from~$L'$ to computable very uniform f-trees such that
$\cT(a)$ is degenerate iff $a = 0_L$.
\end{enumerate}
\end{defn}

Before we define the ordering on~$\bP$, we introduce some further notation:
Given $(\eta, \cT) \in \bP$ and $a \geq_{L'} b$, we need a map $\Theta_{a,b}$
which maps nodes and paths of $\cT(a)$ to nodes and paths of $\cT(b)$. This
will ultimately guarantee the Ziegler reduction between $\mathbf{f}_a$ and
$\mathbf{f}_b$, the degrees corresponding to~$a$ and~$b$ respectively.
		
\begin{notation}\label{def:trTheta}
Let $(\eta, \cT) \in \bP$ and $L' \subseteq L$ be the domain of~$\eta$.
Suppose that $\eta(1_L) = [0,n)$. For $a, b \in L'$ with $a \geq_{L'} b >
0_L$, we have $\eta(a) \supseteq \eta(b)$ are nonempty subsets of $[0,n)$;
say, $\eta(a) = \{i_0, i_1, \ldots, i_{k-1}\}$ and $\eta(b) = \{j_0, j_1,
\ldots, j_{l-1}\} = \{i_{h_0}, i_{h_1}, \ldots, i_{h_{l-1}}\}$, all listed in
increasing order. Then we define
\[
\Theta_{a,b}(\cT(a)(\rho)) =
  \cT(b)\left(\rho[k; h_0, h_1, \ldots, h_{l-1}]\right).
\]
This is a $\subset$-preserving map from~$\ran(\cT(a))$ onto~$\ran(\cT(b))$
which naturally induces a surjection $[\Theta_{a,b}]: [\cT(a)] \rightarrow
[\cT(b)]$. For each $a \in L'$, we define
\[
\Theta_{a,0_L}(\cT(a)(\rho)) = \cT(0_L)(\rho),
\]
which induces the trivial surjection $[\Theta_{a,0_L}]: [\cT(a)] \rightarrow
[\cT(0_L)]$.

For $a \in L'$ and $f_1 = f_{1_L} \in [\cT(1_L)]$, we set $f_a =
[\Theta_{1_L,a}](f_1)$. Similarly, for $\sigma = \sigma_{1_L} =
\cT(1_L)(\rho)$, we define $\sigma_a = \Theta_{1_L,a}(\sigma)$.

Write~$\Theta$ for the collection of such maps corresponding to $(\eta,
\cT)$.
(Analogously, we write~$\Theta^*$ for the collection of such maps
corresponding to $(\eta^*, \cT^*)$, and similarly for other
sub/superscripts.)
\end{notation}
		
\begin{defn}\label{def:partial-order}
The ordering~$\preceq$ on~$\bP$ is defined by letting $(\eta^*, \cT^*)
\preceq (\eta, \cT)$ iff
\begin{enumerate}
\item\label{it:treta*}
$\dom(\eta^*) \supseteq \dom(\eta)$ and for every $a \in \dom(\eta)$, we
have $\eta(a) = \eta^*(a) \cap \eta(1_L)$,
\item\label{it:trT*}
$\cT^*(a) \subseteq \cT(a)$ for every $a \in \dom(\eta)$, and
\item\label{it:trTheta*}
for all $a \geq_{L'} b$ in $L' = \dom(\eta)$, we have $[\Theta^*_{a,b}] =
[\Theta_{a,b}] \restr [\cT^*(a)]$.
\end{enumerate}
\end{defn}

We note the following obvious properties, all immediate by
Notation~\ref{not:[]} and our definition of f-trees:
		
\begin{rem}\label{rem:triple}
In the notation of Definition~\ref{def:triple}, we have:
\begin{enumerate}
\item\label{it:abc>0}
For $a \ge_{L'} b \ge_{L'} c >_{L'} 0_L$, we have $\Theta_{a,c} =
\Theta_{b,c} \circ \Theta_{a,b}$ (when considered as maps on strings).
\item\label{it:abc=0}
For $a \ge_{L'} b \ge_{L'} c \ge_{L'} 0_L$, we have $[\Theta_{a,c}] =
[\Theta_{b,c}] \circ [\Theta_{a,b}]$ (when considered as maps on paths).
\item\label{it:ab_a}
Recall that $\eta(1_L) = [0,n)$. For a (nonzero) join-irreducible element
$a \in L'$, fix~$b_a$ as in	Definition~\ref{def:latt}. Then $\eta(a)
\not\subseteq \eta(b_a)$, and for each $i \in \eta(a) - \eta(b_a)$, each $j
\in \omega$ with $j \equiv i \pmod{n}$ and each $\pi \in 2^j$, we have
$\Theta_{1,b_a}(\cT(1_L)(\catt\pi0)) =
\Theta_{1,b_a}(\cT(1_L)(\catt\pi1))$, but
$\Theta_{1,a}(\cT(1_L)(\catt\pi0)) \neq \Theta_{1,a}(\cT(1_L)(\catt\pi1))$.
\end{enumerate}
\end{rem}

We note that Remark~\ref{rem:triple}\eqref{it:ab_a} takes the place of the
notion of ``alternating f-trees'' in Jacobsen-Grocott/Lempp/Scott
\cite{JLSta}.

The key property of $\Theta_{a,b}$ is the following:
		
\begin{lemma}\label{lem:Theta-is-a-Z-functional}
Let $(\eta, \cT)$ be a condition with $L' = \dom(\eta)$ and $a \geq_{L'} b$
in $L'$. Suppose $f_a \in [\cT(a)]$ and $f_b \in [\cT(b)]$ with
$[\Theta_{a,b}](f_a) = f_b$. Then $f_b \leq_Z f_a$.
\end{lemma}

In fact, the proof will show that indeed $f_b \leq_m f_a$.
		
\begin{proof}
To check whether $x \in f_b$, we first check whether there is $i < 2$ such
that $\sigma(x) = i$ for every $\sigma$ on $\cT(b) \restr x+1$, in which case
$f_b(x) = i$. (Notice that since $\cT(b)$ is computable, this can be ``built
into'' the reduction and thus does require use of the oracle.) Else, there is
some $\rho \in 2^{< \omega}$ such that $f_b(x) = i$ iff~$f_b$ extends
$\cT(b)(\cat \rho i)$. In this case, there is $\pi \in 2^{<\omega}$ with
$\Theta_{a,b}(\cT(a)(\cat \pi i)) = \cT(b)(\cat \rho i)$. Since $\cT(a)$ is
very uniform, there is some~$y$ such that $f_a(y) = i$ iff~$f_a$ extends
$\cT(a)(\cat \pi i)$. Thus, summarizing the argument above, we have that $x
\in f_b$ iff $y \in f_a$. So $f_b \leq_Z f_a$.
\end{proof}
		
\subsection{Four propositions}

Here we lay out the four propositions that form the basis of Lachlan's proof,
adapted to our setting. The first allows us to ``stretch'' out the f-tree so
as to allow extending the finite lattice.
		
\begin{prop}[{following Lachlan \cite[Proposition~1]{La68}}]\label{prop1}
Let $(\eta, \cT)$ be a condition. Then there is a condition $(\eta^*, \cT^*)
\preceq (\eta, \cT)$ such that
\begin{enumerate}
\item\label{prop1eta}
$\dom(\eta^*) = \dom(\eta)$ and $\forall S \in \ran(\eta^*)\, (\text{$|S|$
is even})$ and
\item\label{prop1T}
$\cT^* = \cT$.
\end{enumerate}
\end{prop}
		
\begin{proof}
For~\eqref{prop1eta}, define $\eta^*(a) = \eta(a) \cup \{j + n \mid j \in
\eta(a)\}$ where $n = |\eta(1_L)|$. Then~\eqref{prop1T} follows easily.
To finish, we check that $[\Theta^*_{a,b}] = [\Theta_{a,b}] \restr
[\cT^*(a)]$. Suppose that $|\eta(a)| = k$ and $|\eta(b)| = l$, and $\eta(b)$
consists of the $h_0$th, $\ldots$, $h_{l-1}$th elements of $\eta(a)$. Then
$|\eta^*(a)| = 2k$ and $|\eta^*(b)| = 2l$, and $\eta^*(b)$ consists of the
$h_0$th, $\ldots$, $h_{l-1}$th, $h_l$th, $\ldots$, $h_{2l-1}$th elements of
$\eta^*(a)$. By the definition of~$\eta^*$, for each $q \geq l$ we have that
$h_q = h_{q-l} + k$. Hence, for every string~$\rho$, we have that $\rho[k;
h_0, \ldots h_{l-1}] = \rho[2k; h_0, \ldots h_{l-1}, h_l, h_{2l-1}]$ and
thus, since $\cT^* = \cT$, we have that $[\Theta^*_{a,b}] = [\Theta_{a,b}]
\restr [\cT^*(a)]$.
\end{proof}

We can now specify how we modify the f-tree when we expand the lattice.
		
\begin{prop}[{following Lachlan \cite[Proposition~2]{La68}}]\label{prop2}
Let $(\eta, \cT)$ be a condition such that $\forall S \in \ran(\eta)\,
(\text{$|S|$ is even})$. Let~$L'$ be a simple extension	of~$\dom(\eta)$. Then
there is a condition $(\eta^*, \cT^*) \preceq (\eta, \cT)$ such that
\begin{enumerate}
\item\label{prop2eta}
$\dom(\eta^*) = L'$ and $\eta^* \supseteq \eta$ and
\item\label{prop2T}
$\cT^* \restr \dom(\eta) = \cT$.
\end{enumerate}
\end{prop}
		
\begin{proof}
Say~$L'$ generated by~$\dom(\eta)$ and $a' \in L'$. Let
\begin{align*}
a^- &= \bigvee   \{a \in \dom(\eta) \mid a <_{L'} a'\}\text{ and}\\
a^+ &= \bigwedge \{a \in \dom(\eta) \mid a >_{L'} a'\}.
\end{align*}

Then $|\eta(a^-)|+2 \le |\eta(a^+)|$, so fix $j \in \eta(a^+) - \eta(a^-)$.
Setting $\eta^*(a') = \eta(a^-) \cup \{j\}$ determines~$\eta^*$ uniquely to
satisfy Definition~\ref{def:triple} and~\eqref{prop2eta}. Furthermore, since
$\eta^*(a) = \eta(a)$ and $\cT^*(a) = \cT(a)$ for every $a \in \dom(\eta)$,
we have that $[\Theta^*_{a,b}] = [\Theta_{a,b}] \restr [\cT^*(a)]$ whenever
$a \geq_{L'} b$ in $\dom(\eta)$.
%
%
%
\end{proof}

Next, we show that we can diagonalize or force partiality.
		
\begin{prop}[{following Lachlan \cite[Proposition~3]{La68}}]\label{prop3}
Let $(\eta, \cT)$ be a condition; let $a, b \in \dom(\eta) = L'$ with $a
\nleq_{L'} b$; and let $e, i \in \omega$. Then there is a condition		
$(\eta^*, \cT^*) \preceq (\eta, \cT)$ such that
\begin{enumerate}
\item\label{prop3eta}
$\eta^* = \eta$,
\item\label{prop3T}
for every $c \in \dom(\eta)$, $\cT^*(c) \subseteq \cT(c)$, and
\item\label{prop3neq}
for every $f_1 \in [\cT(1_L)]$, we have $(\Phi_e,\Psi_i)[{f_b}] \neq f_a$.
\end{enumerate}
\end{prop}
		
\begin{proof}
Recall our notation that $\sigma_a = \Theta_{1_L,a}(\sigma)$.
Let $k = |\eta(a)|$ and $l = |\eta(b)|$. First check whether there are $x<k$
and $\sigma \in \ran(\cT(a))$ such that
\begin{itemize}
\item
$x \in \Phi_e[\sigma]$ and $x \in \Psi_i[\sigma]$, or
\item
for all $\tau \in \ran(\cT(a))$ extending~$\sigma$, neither $x \in
\Phi_e[\tau]$ nor $x \in \Psi_i[\tau]$.
\end{itemize}
If so, then fix~$\pi$ with $\sigma = \big(\cT(1_L)(\pi)\big)_a$, where
without loss of generality we may assume that~$|\pi|$ is a multiple of~$n$.
Then we set $\cT^*(1_L)(\rho) = \cT(1_L)(\ca{\pi}{\rho})$ for all $\rho \in
2^{<\omega}$; and now the condition that $[\Theta^*_{1_L,c}] =
[\Theta_{1_L,c}] \restr [\cT^*(1_L)]$ fully determines $\Theta^*(c)$ for all
$c \in L'$.
Note that since~$|\pi|$ is a multiple of~$n$, we have that $[\Theta^*_{c,d}]
= [\Theta_{c,d}] \restr [\cT^*(c)]$ whenever $c \geq_{L'} d$ in~$L'$.
Moreover, now $(\Phi_e,\Psi_i)[f_b](x)$ is not a valid computation for
every~$f_1$ in $[\cT(1_L)]$, so in particular is not equal to $f_a(x)$.

Otherwise, fix $p>0$ and $\rho \in 2^{pl}$ such that
$(\Phi_e,\Psi_i)[\cT(b)(\rho)] \restr |\cT(a)(\str{0^k})|$ is defined, i.e.,
that for each $x < |\cT(a)(\str{0^k})|$, $x \in \Phi_e(\cT(b)(\rho))$ or $x
\in \Psi_i(\cT(b)(\rho))$ (and not both). Since $a \nleq_{L'} b$, there is
some $j \in \eta(a) - \eta(b)$. Hence we can find $\pi', \pi'' \in
2^{<\omega}$ such that $\pi'_b = \pi''_b$ and so
$$
\Theta_{1_L,b}(\cT(1_L)(\pi')) =
\Theta_{1_L,b}(\cT(1_L)(\pi'')) = \cT(b)(\rho).
$$
But $\pi'(j) \neq \pi''(j)$, so $\pi'_a \neq \pi''_a$ and thus
$$
\Theta_{1_L,a}(\cT(1_L)(\pi')) \restr |\cT(a)(\str{0^k})| \neq
\Theta_{1_L,a}(\cT(1_L)(\pi'')) \restr |\cT(a)(\str{0^k})|.
$$
Again we may assume without loss of generality that $|\pi'| = |\pi''|$ is a
multiple of~$n$. And now we can set $\pi = \pi'$ or $\pi = \pi''$ such that
$\sigma_1 = \cT(1_L)(\pi)$ while ensuring that $\sigma_b = \cT(b)(\rho)$ as
well as
$$
(\Phi_e,\Psi_i)[\sigma_b] \restr |\cT(a)(\str{0^k})|
\neq \sigma_a \restr |\cT(a)(\str{0^k})|,
$$
which will establish~\eqref{prop3neq}.
		
As before, we set $\cT^*(1_L)(\nu) = \cT(1_L)(\ca{\pi}{\nu})$ for all $\nu
\in 2^{<\omega}$; and again the condition that $[\Theta^*_{1_L,c}] =
[\Theta_{1_L,c}] \restr [\cT^*(1_L)]$ fully determines $\cT^*(c)$ for all $c
\in L'$.
This implies~\eqref{prop3T} and $[\Theta^*_{a,b}] = [\Theta_{a,b}] \restr
[\cT^*(a)]$, using that~$|\pi|$ is a multiple of~$n$.
\end{proof}

The last proposition is the core of the argument, allowing us to conclude
that we can force the degrees below~$\mathbf{f}$ (the path we are building in
$\cT(1_L)$) to form exactly the lattice~$L$ -- and not a superstructure.
		
		
\begin{prop}[{following Lachlan \cite[Proposition~4]{La68}}]\label{prop4}
Let $(\eta, \cT)$ be a condition with $\dom(\eta) = L'$, and let $e, i \in
\omega$. Then there are a condition $(\eta^*, \cT^*) \preceq (\eta, \cT)$ and
$a_0 \in L'$ such that
\begin{enumerate}
\item\label{prop4eta}
$\eta^* = \eta$,
\item\label{prop4T}
for every $a \in \dom(\eta)$, $\cT^*(a) \subseteq \cT(a)$, and
\item\label{prop4eq}
there is $a_0 \in L$ such that for every $f_1 \in [\cT^*(1_L)]$, if
$(\Phi_e,\Psi_i)[f_1]$ is a valid computation, then $(\Phi_e,\Psi_i)[f_1]
\equiv_Z f_{a_0}$.
\end{enumerate}
\end{prop}
		
\begin{proof}
Fix~$n$ so that $\eta(1_L) = [0,n)$.

We distinguish two cases: If there are $x \in \omega$ and $\pi \in
2^{<\omega}$ such that, for $\sigma = \cT(1_L)(\pi)$,
\begin{itemize}
\item
$x \in \Phi_e[\sigma]$ and $x \in \Psi_i[\sigma]$, or
\item
for all $\tau \in \ran(\cT(a))$ extending~$\sigma$, neither $x \in
\Phi_e[\tau]$ nor $x \in \Psi_i[\tau]$,
\end{itemize}
then we may again assume without loss of generality that~$|\pi|$ is a
multiple of~$n$.
As before, we set $\cT^*(1_L)(\rho) = \cT(1_L)(\ca{\pi}{\rho})$ for all $\rho
\in 2^{<\omega}$; and again the condition that $[\Theta^*_{1_L,a}] =
[\Theta_{1_L,a}] \restr [\cT^*(1_L)]$ fully determines $\cT^*(a)$ for all $a
\in L'$.

Otherwise, we will first need to refine some notions from
Definition~\ref{def:ftree} as follows:
			
\begin{defn}\label{def:treerefd}
Suppose that $\tau = \cT(1_L)(\pi)$ and $\tau' = \cT(1_L)(\pi')$ such that
$|\tau| = |\tau'|$, and that~$\tau$ and~$\tau'$ $\str{e,i}$-split $\sigma =
\cT(1_L)(\rho)$ on~$\cT(1_L)$.
				
Then we say~$\tau$ and~$\tau'$ \emph{differ in} $a \in L'$ if $\tau_a =
\Theta_{1_L,a}(\tau) \neq \Theta_{1_L,a}(\tau') = \tau'_a$ (and so are
incomparable since they have the same length); otherwise,~$\tau$ and~$\tau'$
\emph{agree in} $a \in \eta(1_L)$. We say that~$\tau$ and~$\tau'$
\emph{differ minimally in} $a \in L'$ if they differ in $a \in L'$ and
this~$a$ is minimal such in~$L'$.

For each $\sigma \in \ran(\cT(1_L))$, we set
\begin{multline*}
\quad \cC(\sigma) = \{a \in L' \mid \\
		\exists \tau, \tau'\,
(\text{$\tau, \tau'$ minimally $\str{e,i}$-split~$\sigma$ on~$\cT(1_L)$
and minimally differ in~$a$})\}.
\end{multline*}
\end{defn}

Then we have the following
			
\begin{claim}\label{cl:Csigma}
\begin{enumerate}
\item\label{it:Clim}
$\sigma \subset \tau$ in $\ran(\cT(1_L))$ implies $\cC(\sigma) \supseteq
\cC(\tau)$.
\item\label{it:Cstable}
There is $\sigma_0 \in \ran(\cT(1_L))$ such that $\cC(\sigma_0) =
\cC(\sigma)$ for all $\sigma \supseteq \sigma_0$ in $\ran(\cT(1_L))$.
\item\label{it:Cnot}
Let~$\sigma_0$ be as in \eqref{it:Cstable}, and $a_0 = \bigvee
\cC(\sigma_0)$ (with $a_0 = 0_L$ if $\cC(\sigma_0) = \emptyset$). Then
there are no $\tau, \tau' \supset \sigma_0$ with $\tau, \tau' \in					
\ran(\cT(1_L))$, $|\tau| = |\tau'|$, $\tau_{a_0} = \tau'_{a_0}$, and which
$\str{e,i}$-split~$\sigma$ on~$\cT(1_L)$.
\end{enumerate}
\end{claim}
			
\begin{proof}
Item~\eqref{it:Clim} is obvious, and so is~\eqref{it:Cstable} since~$L'$ is
finite.
				
For~\eqref{it:Cnot}, assume for a contradiction that there are~$\tau$
and~$\tau'$ with $\tau_{a_0} = \tau'_{a_0}$, and such that~$\tau$ and~$\tau'$
$\str{e,i}$-split~$\sigma_0$ on $\cT(1_L)$. By Lemma~\ref{lem:minsplit}, we
may in addition assume that they minimally~$\str{e,i}$-split~$\sigma_0$ via
some $j < n$. Let $a \in L'$ be such that~$\tau$ and~$\tau'$ differ minimally
in~$a$. Then $j \in \eta(a)$, but this implies $a \in \cC(\sigma_0)$. Since
$\tau_{a_0} = \tau'_{a_0}$, we have $j \notin \eta(a_0)$. So $a \nleq_{L'}
a_0$, contradicting our choice of~$a_0$. \qedhere
%
\end{proof}

We can now define~$\cT^*$, using the values of~$a_0$ and~$\sigma_0$ above,
where we again assume that~$|\pi|$ is a multiple of~$n$ for $\sigma_0 =
\cT(1_L)(\pi)$. We first set $\cT^*(1_L)(\str{}) = \sigma_0$. We then proceed
by recursion on $m = l \cdot n + j$: Assume that we have already defined
$\cT^*(1_L)(\rho)$ for all $\rho \in 2^m$, and that for any such~$\rho$,
we have defined $\cT^*(1_L)(\rho) = \cT(1_L)(\pi(\rho))$, where~$|\pi(\rho)|$
is a multiple of~$n$. We now distinguish two cases:
			
\subsubsection*{Case~1: $j \notin \eta(a_0)$:}

Then we set, for all $\rho \in 2^m$ and all $k<2$,
$$
\cT^*(1_L)(\catt{\rho}{k}) =
\cT(1_L)(\catttt{\pi(\rho)}{0^j}{k}{0^{n-j-1}}).
$$
			
\subsubsection*{Case~2: $j \in \eta(a_0)$:}

Then $j \in \eta(a)$ for some $a \in \cC(\sigma_0)$. Hence there are~$\pi_0$
and~$\pi_1$ such that $\cT(1_L)(\ca{\pi(\str{0^m})}{\pi_0})$ and
$\cT(1_L)(\ca{\pi(\str{0^m})}{\pi_1})$ minimally $\str{e,i}$-split
$\cT^*(\str{0^m})$ in $\cT(1_L)$ and differ minimally in~$a$. Extending if
necessary, we may further assume that $|\pi_0| = |\pi_1|$ is divisible by
$n$.
			
We then set, for all $\rho \in 2^m$ and all $k<2$,
\[
\cT^*(1_L)(\catt{\rho}{k}) =
\cT(1_L)(\cattttt{\pi(\rho)}{\pi_k}{0^j}{k}{0^{n-j-1}}).
\]

And, again, the condition that $[\Theta^*_{1_L,a}] = [\Theta_{1_L,a}] \restr
[\cT^*(1_L)]$ fully determines $\cT^*(a)$ for all $a \in L'$.
			
Since~$\cT$ was very uniform,~$\cT^*$ remains very uniform.

We now first establish, for Case~2, the following
			
\begin{claim}
For all $\rho \in 2^m$, we have that $\cT^*(1_L)(\catt{\rho}{0})$ and
$\cT^*(1_L)(\catt{\rho}{1})$ $\str{e,i}$-split $\cT^*(1_L)(\rho)$
on~$\cT^*(1_L)$ and differ in~$a_0$.
\end{claim}
			
\begin{proof}
By hypothesis, we have that $\cT(1_L)(\ca{\pi(\str{0^m})}{\pi_0})$
and $\cT(1_L)(\ca{\pi(\str{0^m})}{\pi_1})$
$\str{e,i}$-split $\cT^*(1_L)(\str{0^m})$ in $\cT(1_L)$. Thus there is~$x$
with $x \in \Psi_i[\cT(1_L)(\ca{\pi(\str{0^m})}{\pi_0})]$ and $x \in
\Phi_e[\cT(1_L)(\ca{\pi(\str{0^m})}{\pi_1})]$.
But then $x \in \Psi_i[\cT(1_L)(\ca{\pi(\str{0^m})}{\pi_0})]$ hinges only on
$(\ca{\pi(\str{0^m})}{\pi_0})(j') = 0$ for some $j' \in		
[|\pi(\str{0^m})|, |\cat{\pi(\str{0^m})}{\pi_0}|)$. Now this ensures that
$x \in \Psi_i[\cT(1_L)(\ca{\pi(\rho)}{\pi_0})]$ and $x \in
\Phi_e[\cT(1_L)(\ca{\pi(\rho)}{\pi_1})]$ for all $\rho \in 2^m$, so
$\cT^*(1_L)(\cat{\rho}{0})$ and $\cT^*(1_L)(\cat{\rho}{1})$ $\str{e,i}$-split
$\cT^*(1_L)(\rho)$ on~$\cT(1_L)$ and thus also on~$\cT^*(1_L)$.

Our definition of $\cT^*(1_L)(\catt{\rho}{k})$ for $k<2$ now also guarantees
that they differ in~$a_0$.
%
\end{proof}

Concluding the proof of Proposition~\ref{prop4}, we now observe that given
the definition of~$\cT^*(1_L)$, item~\eqref{prop4eta} of
Proposition~\ref{prop4} completely determines~$\eta^*$. Furthermore, we
define, for $a \in L'$ and $\rho \in 2^{<\omega}$, $\cT(a)(\rho) =
\Theta_{1,a}(\cT^*(1_L)(\rho))$. This immediately implies item~\eqref{prop4T}
and that $[\Theta_{a,b}^*] = [\Theta_{a,b}] \restr [\cT^*(a)]$ whenever $a
\geq_{L'} b$ in~$L'$.

It remains to verify item~\eqref{prop4eq} of Proposition~\ref{prop4}. So fix
$f_1 \in [\cT^*(1_L)]$ such that $(\Phi_e,\Psi_i)[f_1]$ is a valid
computation.

We first verify that $(\Phi_e,\Psi_i)[f_1] \le_Z f_{a_0}$, so fix an
argument~$x$. Effectively in~$f_{a_0}$, we can find $\rho \in 2^{<\omega}$
such that, for $\tau = \cT(1_L)(\rho)$:
\begin{itemize}
\item
$\tau_{a_0} = (f_{a_0} \restr |\tau|)_{a_0}$; and
\item
$x \in \Phi_e[\tau]$ or $x \in \Psi_i[\tau]$.
\end{itemize}
Then, by Claim~\ref{cl:Csigma}\eqref{it:Cnot}, $(\Phi_e,\Psi_i)[f_1](x) =
(\Phi_e,\Psi_i)[\tau](x)$ as desired.

On the other hand, let's now verify that $f_{a_0} \le_Z
(\Phi_e,\Psi_i)[f_1]$. Then~$f_{a_0}$ at an argument~$x'$, say, is determined
by~$f_1$ at an argument~$x$.
			

If there is $k<2$ and some~$\ell$ such that every $\sigma \in \ran(\cT^*)$ of
length~$\ell$ has $\sigma(x) = k$, then $f_1(x) = k$ and so $f_{a_0}(x') =
k$. Otherwise, there are $\sigma, \tau \in \ran(\cT^*)$ such that $\sigma(x)
= 0 < 1 = \tau(x)$. Since~$\cT^*$ is very uniform, we may in addition assume
that there is $\rho \in 2^{< \omega}$ such that $\cT^*(\cat \rho 0) = \sigma$
and $\cT^*(\cat \rho 1) = \tau$. Now, by the way we constructed~$\cT^*$ (in
Case~2 of the proof of Proposition~\ref{prop4}), we know that $(\Phi_e,
\Psi_i)[\sigma](y) = 0 < 1 (\Phi_e, \Psi_i)[\tau](y)$ for some~$y$, and so
$f_{a_0}(x') = f_1(x) = (\Phi_e, \Psi_i)[f_1](y)$ can be Ziegler-computed from
$(\Phi_e, \Psi_i)[f_1]$ as claimed.
\end{proof}
		
\subsection{Proof of Theorem~\ref{thm:ctdl}}

Let the countable distributive lattice~$L$ be given by a (usually
noncomputable) approximation by finite distributive lattices~$L_s$ such
that~$L_0$ is the 2-element lattice, and each~$L_{s+1}$ either equals~$L_s$
or else is a simple extension of~$L_s$. We will define a (usually
noncomputable) decreasing sequence of conditions $\{(\eta^s, \cT^s)\}_{s \in
\omega}$; the unique infinite path~$f_1$ through all f-trees~$\cT^s$ will
represent the desired Ziegler degree.

We also define a (possibly non-effective) enumeration $\{(a_s, b_s, e_s,
i_s)\}_{s \in \omega}$ of all $a, b \in L$ and $e, i \in \omega$ such that
$a, b \in L_s$, $a \neq 1_L$ is join-irreducible in~$L_s$, and $b = b_a$ (in
the sense of~$L_s$, and so $b \in L_s$).
		
\subsubsection*{Stage $s = 0$:}

Let $\eta^0_{1_{L_0}}: L_0 \to \cP(\omega)$ be given by
$\eta^0_{L_0}(0_{L_0}) = \emptyset$ and $\eta^0_{L_0}(1_{L_0}) = \{0\}$. Let
$\cT^0(1_L)$ be the full binary f-tree and $\cT^0(0_L)$ be the degenerate
f-tree.
		
\subsubsection*{Stage $4s+1$:}

We obtain condition $(\eta^{4s+1}, \cT^{4s+1})$ from $(\eta^{4s}, \cT^{4s})$
using Proposition~\ref{prop1}.
		
\subsubsection*{Stage $4s+2$:}

We obtain condition $(\eta^{4s+2}, \cT^{4s+2})$ from $(\eta^{4s+1},
\cT^{4s+1})$ using Proposition~\ref{prop2} applied to the lattices~$L_s$
and~$L_{s+1}$.
		
\subsubsection*{Stage $4s+3$:}

We obtain condition $(\eta^{4s+3}, \cT^{4s+3})$ from $(\eta^{4s+2},
\cT^{4s+2})$ using Proposition~\ref{prop3} applied to the Ziegler reduction
$(\Phi_e, \Psi_i)$ and $a, b \in L$ where $(a, b, e, i) = (a_s, b_s, e_s,
i_s)$.
		
\subsubsection*{Stage $4s+4$:}

We obtain condition $(\eta^{4s+4}, \cT^{4s+4})$ from $(\eta^{4s+3},
\cT^{4s+3})$ using Proposition~\ref{prop4} applied to the Ziegler reduction
$(\Phi_e, \Psi_i)$ where $\str{e,i} = s$.
\smallskip
		
\subsubsection*{Verification}

We start by showing that if $a \geq_L b$, then $\mathbf{f}_b \leq_Z
\mathbf{f}_a$. Indeed, fix a stage $t = 4s+2$ such that $a,b \in L_s$.
Then~$f_a$ is on $\cT^t(a)$ and~$f_b$ is on $\cT^t(b)$, which are both
computable very uniform f-trees, and, using \eqref{it:trTheta*} of
Definition~\ref{def:partial-order}, we have $f_b = [\Theta^t_{a,b}](f_a)$.
Hence Lemma~\ref{lem:Theta-is-a-Z-functional} implies that $f_b \leq_Z f_a$.


At stages $s \equiv 2 \pmod 4$, we guarantee that each $a \in L$ is
introduced to the construction at some stage. At stages $s \equiv 3 \pmod 4$,
we ensure that $a \nleq_L b$ implies that $f_a \nleq^* f_b$. Combined with
the previous paragraph, this guarantees that there is an embedding of the
lattice~$L$ into $\cD_Z(\leq \mathbf{f}_1)$. At stages $s \equiv 4 \pmod 4$,
we force that if $\mathbf{a} \leq_Z \mathbf{f}_1$, then $\mathbf{a} \equiv_Z
\mathbf{f}_a$ for some $a \in L$. This completes the proof of
Theorem~\ref{thm:ctdl}.
\end{proof}

\end{document}